\documentclass[12pt]{amsart}
  
\usepackage{amsmath,amsthm,amssymb,mathrsfs}
 
\numberwithin{equation}{section}

\renewcommand{\geq}{\geqslant}
\renewcommand{\leq}{\leqslant}
\newcommand{\Osh}{{\mathcal O}}                        

\renewcommand{\H}{\mathrm{H}}                          

\newcommand{\G}{\mathrm{G}}

\newcommand{\kk}{\mathbf{k}}

\newcommand{\Vol}{\operatorname{Vol}}

\newcommand{\KK}{\mathbf{K}}

\newcommand{\CC}{\mathbb{C}} 
\newcommand{\PP}{\mathbb{P}} 
\newcommand{\QQ}{\mathbb{Q}} 
\newcommand{\RR}{\mathbb{R}} 
\newcommand{\ZZ}{\mathbb{Z}} 

\newtheorem{theorem}{Theorem}[section]
\newtheorem{lemma}[theorem]{Lemma}

\theoremstyle{definition}
\newtheorem{defn}[theorem]{Definition}

\newtheorem{example}[theorem]{Example}

\begin{document}
\title[DH-measure for filtered linear series]
{On Duistermaat-Heckman measure for filtered linear series}
\author{Nathan Grieve}
\address{Department of Mathematics \& Computer Science,
Royal Military College of Canada, P.O. Box 17000,
Station Forces, Kingston, ON, K7K 7B4, Canada
}
\address{
School of Mathematics and Statistics, 4302 Herzberg Laboratories, Carleton University, 1125 Colonel By Drive, Ottawa, ON, K1S 5B6, Canada
}
\address{D\'{e}partement de math\'{e}matiques, Universit\'{e} du Qu\'{e}bec \`a Montr\'{e}al, Local PK-5151, 201 Avenue du Pr\'{e}sident-Kennedy, Montr\'{e}al, QC, H2X 3Y7, Canada}
\email{nathan.m.grieve@gmail.com}

\begin{abstract}  
We revisit work of S. Boucksom, C. Favre, and M. Jonsson (J. Algebraic Geom. \textbf{18} (2009), no. 2, 279--308);  Boucksom and H. Chen (Compos. Math. \textbf{147} (2011), no. 4, 1205--1229); and S. Boucksom, A. K\"{u}ronya, C. Maclean, and T. Szemberg (Math. Ann. \textbf{361} (2015), no.~3--4, 811--834).  
The key point is to associate a Duistermaat-Heckman measure to a filtered big linear series on a given projective variety.  The expectation of the measure admits a description via the theory of Newton-Okounkov bodies.  Such considerations have origins in symplectic geometry.  They have applications for $\mathrm{K}$-stability and Diophantine arithmetic geometry of projective varieties. 
 \\

Nous revisitons les travaux de S. Boucksom, C. Favre, and M. Jonsson (J. Algebraic Geom. \textbf{18} (2009), no. 2, 279--308);  Boucksom et H. Chen (Compos. Math. \textbf{147} (2011), no. 4, 1205--1229); et S. Boucksom, A. K\"{u}ronya, C. Maclean, et T. Szemberg (Math. Ann. \textbf{361} (2015), no.~3--4, 811--834).  Nous \'{e}tudions deux r\'{e}sultats, qui sont \`{a} l'intersection de la $\mathrm{K}$-stabilit\'{e}, de la g\'{e}om\'{e}trie arithm\'{e}tique Diophantienne, et de la th\'{e}orie des corps convexe de Newton-Okounkov. Ils se rapportent \`{a}  des filtrations de grands syst\`{e}mes lin\'{e}aire sur des vari\'{e}t\'{e}s projectives et \`{a} l'existence de une mesure de Duistermaat-Heckman.   La mesure de Duistermaat-Heckman vient de la g\'{e}om\'{e}trie symplectique. L'esp\'{e}rance de la mesure peut \^{e}tre calcul\'{e}e par la th\'{e}orie de Newton-Okounkov.   
\end{abstract}
\thanks{\emph{Mathematics Subject Classification (2020):} Primary 14C20; Secondary 13A18.}
\thanks{Keywords:  Duistermaat-Heckman measure; Newton-Okounkov body; Restricted volume functions.}
\thanks{I hold grants DGECR-2021-00218 and RGPIN-2021-03821 from the Natural Sciences and Engineering Research Council of Canada.}

\maketitle

\section{Introduction}

Let $L$ be a big line bundle over an irreducible projective variety $X$. 
Our purpose here, is to revisit \cite[Theorem 1.11]{Boucksom:Chen:2011} and \cite[Theorem 2.24]{Boucksom:Kuronya:Maclean:Szemberg:2015}, which establish existence of a \emph{Duistermaat-Heckman measure} with respect to certain  \emph{linearly bounded above}  and \emph{pointwise bounded below} \emph{$\RR$-filtrations} of the section ring 
$$
R = R(X,L) := \bigoplus_{m\geq 0} \H^0(X,mL) = \bigoplus_{m \geq 0} R_m \text{.}
$$
A key tool is the theory of Newton-Okounkov bodies (\cite{Khovanskii:1992}, \cite{Okounkov:2003}, \cite{Lazarsfeld:Mustata:2009} and \cite{Kaveh:Khovanskii:2012}).

Here, and in what follows, we work over a fixed algebraically closed characteristic zero base field $\kk$. 
Further, we recall that a line bundle $L$ on an irreducible projective variety $X$ is \emph{big} if its \emph{volume}
$$
\Vol(L) := \limsup_{m \to \infty} \frac{h^0(X,mL)}{m^d/d!} \text{,}
$$
for
$$
d := \dim X
$$
and
$$
h^0(X,mL) := \dim \H^0(X,mL) \text{,}
$$
is nonzero.  We refer to \cite[Sections 2.2A and 2.2.C]{Laz} for more details about big line bundles and their volumes.

In this article, by a Duistermaat-Heckman measure, we mean a certain  singular probability measure, with compact support, on the real number line.  It arises as the limit of discrete measures that are associated to certain kinds of filtered linear series.  

To formulate our main theorems, let $R$ be the section ring of $L$ a big line bundle on an irreducible projective variety $X$ over an algebraically closed characteristic zero field $\kk$.  Moreover, fix  a  \emph{decreasing}, \emph{left-continuous}, \emph{pointwise bounded below} and  \emph{linearly bounded above}  \emph{multiplicative}  \emph{$\RR$-filtration }
$$\mathcal{F}^\bullet = \mathcal{F}^\bullet R$$ 
of $R$ (see Definition \ref{linearly:bounded:above}).  The \emph{vanishing numbers} of $\mathcal{F}^\bullet $, for $m \in \ZZ_{\geq 0}$, are the sequence of real numbers
$$
a_{\min}(mL) = a_0(mL) \leq \dots \leq a_{n_m}(mL) = a_{\max}(mL) \text{,}
$$
which are defined by the condition that
$$
a_j(mL) := 
\inf 
\{ 
t \in \RR : \operatorname{codim} \mathcal{F}^t \H^0(X, mL) \geq j+1
\} \text{.}
$$
The discrete measures that they determine are described as
\begin{equation}\label{DH:measures:intro:eqn}
\nu_m = \nu_m(t) := \frac{1}{h^0(X,mL)} \sum_{j = 0}^{n_m} \delta_{m^{-1} a_j(mL)}(t)\text{.}
\end{equation}

In what follows, we establish Theorems \ref{DH:Expect} and \ref{DH:restricted:volume:thm} below.  A key tool is the theory of Newton-Okounkov bodies, \cite{Lazarsfeld:Mustata:2009}.  Another is the differentiability of the volume function, \cite[Theorem A]{Boucksom:Favre:Jonsson:2009}.

Theorem \ref{DH:Expect}, below, encompasses \cite[Theorem 1.11]{Boucksom:Chen:2011}.  In our formulation, we state explicitly the fact that the limit expectation of the measures \eqref{DH:measures:intro:eqn} coincide with the expectations of the limit measure.  (Compare with \cite[Remark 1.12 and Corollary 1.13]{Boucksom:Chen:2011} and \cite[Section 3.3]{Grieve:2018:autissier}.)

\begin{theorem}
\label{DH:Expect}
Let $\mathcal{F}^\bullet = \mathcal{F}^\bullet R$ be a decreasing, left-continuous,  pointwise bounded below and linearly bounded above multiplicative $\RR$-filtration of
the section ring
$$
R = R(X,L) := \bigoplus_{m\geq 0} \H^0(X,mL) = \bigoplus_{m \geq 0} R_m \text{,}
$$
for $L$ a big line bundle on an irreducible projective variety $X$.  Then, the corresponding discrete measures \eqref{DH:measures:intro:eqn},
which are determined by $\mathcal{F}^\bullet$, 
converge weakly to  a limit measure
$$
\nu := \lim_{m \to \infty} \nu_m \text{.}
$$
Furthermore, the limit of the expectations of the measures $\nu_m$ coincides with the expectation of the limit measure $\nu$.  In other words 
$$
\mathbb{E}(\nu) = \lim_{m \to \infty} \mathbb{E}(\nu_m) \text{.}
$$
\end{theorem}

The following Theorem \ref{DH:restricted:volume:thm} includes \cite[Theorem 2.24]{Boucksom:Kuronya:Maclean:Szemberg:2015}.  
For example, as mentioned in \cite[Theorem 2.24]{Boucksom:Kuronya:Maclean:Szemberg:2015}, the underlying projective variety need not be nonsingular.  On the other hand, here, we recall, following closely the approach of \cite[Section 4.2]{Boucksom:Favre:Jonsson:2009}, a few additional details for establishing the more general result.  

Another aspect to the proof of Theorem \ref{DH:restricted:volume:thm} uses the differentiability property of the volume function (\cite[Theorem A]{Boucksom:Favre:Jonsson:2009}).  We also formulate  Theorem \ref{DH:restricted:volume:thm} in terms of divisors over the given normal variety as opposed to the formulation of divisorial valuations as is done in \cite{Boucksom:Kuronya:Maclean:Szemberg:2015}.  Finally, note that with the  nonsingularity assumption   \cite[Theorem 2.24]{Boucksom:Kuronya:Maclean:Szemberg:2015} follows after applying \cite[Corollary C]{Boucksom:Favre:Jonsson:2009}.  Here we note that, in the more general context, the same conclusion follows as a consequence of \cite[Theorem A]{Boucksom:Favre:Jonsson:2009}.

\begin{theorem}
\label{DH:restricted:volume:thm}
Let $L$ be a big line bundle on a normal irreducible projective variety $X$ and $E$ an irreducible and reduced nonzero Cartier divisor supported on some normal proper model of $X$.  In particular, $E$ is a divisor over $X$.
Fix such a model 
$$
\pi \colon Y \rightarrow X
$$
with the property that 
$$E \subseteq Y $$  
and let
$\mathcal{F}^\bullet $ 
be the corresponding filtration of the section ring that is determined by $E$.  In this context, the limit measure $\nu$ may be described as
$$
\nu = \nu_{L,E} =  \frac{\Vol_{X|E}(\pi^*L - t E)}{\Vol(L)/\operatorname{dim} X } \mathrm{d}t \text{.}
$$
\end{theorem}

We prove Theorem \ref{DH:Expect} in Section \ref{DH:Proof} and  Theorem \ref{DH:restricted:volume:thm} in Section \ref{restricted:volume:functions}.

To place matters into their proper context, it is important to note that the theory of Duistermaat-Heckman measure, and its moments, appears in several 
contexts.  For example, there are  applications for test configurations and questions that surround the $\mathrm{K}$-stability of polarized projective varieties.
They build on earlier work of Duistermaat and Heckman, \cite{Duistermaat:Heckman:1982}, and 
Donaldson, \cite{Donaldson:2002}.
In more recent times, the structure of the distributions that surround these measures has been clarified in \cite{Boucksom:Chen:2011}, \cite{Boucksom:Kuronya:Maclean:Szemberg:2015} and \cite{Boucksom-Hisamoto-Jonsson:2016}.
 
It is also helpful to recall, as in \cite{Okounkov:1996}, the manner in which Duistermaat-Heckman measure arises within the traditional realm of Mumford's Geometric Invariant Theory.  Consider, for example, the case of a connected reductive group $\G$ acting on a $\G$-linearized polarized irreducible projective variety $(X,L)$.  By linearity of the $\G$-action, the section ring
$$
R = R(X,L)
$$
admits a decomposition as a $\G$-module
\begin{equation}\label{section:ring:weight:decomp}
R = \bigoplus_{m \geq 0} \H^0(X,mL) = \bigoplus_{m \geq 0} \bigoplus_{\lambda} V_\lambda^{\oplus \mu_m(\lambda)} \text{.}
\end{equation}
In the decomposition \eqref{section:ring:weight:decomp}, $V_\lambda$ is the irreducible $\G$-module with highest weight $\lambda$; furthermore, $\mu_m(\lambda)$ is the multiplicity of $V_\lambda$ in $\H^0(X,mL)$.

Within this setting,  existence of a Duistermaat-Heckman measure allows for an asymptotic quantitative study of the extent to which the multiplicities $\mu_m(\lambda)$ equidistribute.

Returning to our specific focus here, namely the general study of filtered linear series (and filtrations of section rings), as explained in \cite[Section 1]{Boucksom:Kuronya:Maclean:Szemberg:2015} and \cite[Section 1]{Grieve:2018:autissier}, let us mention that some motivation arises from the study of Weierstrass points, and more generally vanishing sequences associated to linear series, on curves. 

Moreover, establishing an analogue of filtered linear series, within the framework of adelically normed graded linear series is a key feature of the works \cite{Boucksom:Chen:2011} and \cite{Chen:2010}. Both \cite{Chen:2010} and \cite{Boucksom:Chen:2011} are motivated, in part, by T.~Fujita's approximation theorem, for Zariski decompositions, \cite{Fujita:1994}. They obtain analogous results within the adelic (Arakelov) setting.   

For completeness, let us recall these adelic results, from \cite{Boucksom:Chen:2011}, in some detail.  Indeed, working over a base number field $\KK$, let $X$ be a $d$-dimensional projective variety and $L$ a big line bundle on $X$.  Over the ring of $\KK$-integers, let $\mathcal{X}$ be a flat projective model of $X$ and let $\mathcal{L}$ be a line bundle on $\mathcal{X}$ and extending $L$. 

Given $|\cdot|$, a conjugate invariant continuous Hermitian metric on the complexification of $L$, consider, for each $m \in \ZZ_{\geq 0}$, the finite sets of \emph{small sections}
$$
\hat{\H}(m\overline{L}) := \left\{ s \in \H^0(\mathcal{X},m\mathcal{L}) : \sup_{X(\CC)} |s| \leq 1 \right \}\text{.}
$$

In \cite{Boucksom:Chen:2011}, developing further the viewpoint of Yuan \cite{Yuan:2009}, Boucksom and Chen establish existence of an \emph{arithmetic Okounkov body}
$$
\hat{\Delta}(\overline{L}) \subseteq \RR^{d+1} \text{.}
$$
It is defined as
$$
\hat{\Delta}(\overline{L}) := \{(x,t) \in \Delta(L) \times \RR : 0 \leq t \leq G_{\mathcal{F}_{\min}}(x) \} \text{.}
$$
Here, $\Delta(L)$ is the Newton-Okounkov body of $L$ and  $G_{\mathcal{F}_{\min}}$ is the \emph{concave transform} of the adelic minima filtration.
 
A special case of \cite[Theorem 2.8]{Boucksom:Chen:2011}, is then that the volume of this compact convex set, $\hat{\Delta}(\overline{L})$, may be described as
$$
\Vol \hat{\Delta}(\overline{L}) 
= \lim_{m \to \infty} \frac{\log \# \hat{\H}(m\overline{L})}{m^{d+1}[\KK:\QQ]}  \text{.} 
$$
Further, this \emph{normalized arithmetic volume}, $\Vol \hat{\Delta}(\overline{L})$, 
can be approximated by the arithmetic volumes of finitely generated subseries.

Finally, it is important to mention the more recent progress, which has arisen within the areas of $\mathrm{K}$-stability and higher dimensional Diophantine approximation (for projective varieties).  The key objects of interest are invariants that are associated to those filtrations of linear series which arise via orders of vanishing along divisorial valuations and subschemes.  

From the perspective of $\mathrm{K}$-stability,  this was emphasized by Donaldson, \cite{Donaldson:2002}, and Ross-Thomas, \cite{Ross:Thomas:2006} and \cite{Ross:Thomas:2007}.  For the case of $\mathrm{K}$-stability for $\QQ$-Fano varieties, a key result is K.~Fujita's  valulative criterion, \cite{Fujita:2019}.  It expands on earlier work of C. Li, \cite{Li:2017a}, Li and Xu, \cite{Li:Xu:2014}, and G.~Tian \cite{Tian:1997}.
Here, we place emphasis on the connection to Newton-Okounkov bodies.

A detailed account of the algebraic theory of $\mathrm{K}$-stability has been given by Boucksom-Hisamoto-Jonsson, \cite{Boucksom-Hisamoto-Jonsson:2016}, while the foundations of the theory of Newton-Okounkov bodies is due to Khovanskii, Okounkov, Lazarsfeld-Musta{\c{t}}{\v{a}} and Kaveh-Khovanskii. (See \cite{Khovanskii:1992}, \cite{Okounkov:2003}, \cite{Lazarsfeld:Mustata:2009} or \cite{Kaveh:Khovanskii:2012}.)  

This theory was expanded upon to treat the case of reductive group actions in \cite{Kaveh:Khovanskii:2012b}.  That the theory of Newton-Okounkov bodies has applications to the study of filtrations on linear series, is a key point of \cite{Boucksom:Chen:2011}, \cite{Witt-Nystrom:2012} and \cite{Boucksom:Kuronya:Maclean:Szemberg:2015}.  A recent extension to the theory of Newton-Okounkov bodies, to treat the case of abstract subsemigroups of $\ZZ^n$, is given in \cite{Kuronya:Maclean:Roe:2021b} and \cite{Kuronya:Maclean:Roe:2021a}.

In Section \ref{OK:bodies:concave:transform}, we summarize  the  construction of Newton-Okounkov bodies and concave transform for linearly bounded above and pointwise bounded below $\RR$-filtrations of section rings.   In Section \ref{Preliminiaries:Filtered:Algebras}, we discuss our conventions for filtered linear series.  We place some emphasis on those filtrations that arise via orders of vanishing along a divisor. (See Example \ref{orders:vanishing}.)

More recently, it has been discovered that Duistermaat-Heckman measures and the theory of the Newton-Okounkov bodies, for linear series on polarized projective varieties, arise naturally upon considering Diophantine arithmetic questions.  A collection of works that surround this perspective include \cite{Faltings:Wustholz}, \cite{Evertse:Ferretti:2002}, \cite{Ferretti:2003}, \cite{McKinnon-Roth}, \cite{Ru:Vojta:2016}, \cite{Ru:Vojta:2021}, \cite{Ru:Wang:2016}, \cite{Grieve:Function:Fields}, \cite{Grieve:2018:autissier}, \cite{Grieve:chow:approx},  \cite{Heier:Levin:2017}, \cite{Grieve:toric:gcd:2019}, \cite{Grieve:Divisorial:Instab:Vojta} and \cite{Grieve:points:bounded:degree} (among others).  

Elaborating on the interactions amongst these Diophantine arithmetic results and those that within the area of $\mathrm{K}$-stability and the theory of Newton-Okounkov bodies, was one purpose of \cite{Grieve:2018:autissier}.  In particular, it builds on earlier work of Faltings and W\"{u}stholz, \cite{Faltings:Wustholz}, R.~G. Ferretti, \cite{Ferretti:2000} and \cite{Ferretti:2003}, McKinnon-Roth, \cite{McKinnon-Roth}, and Ru-Vojta, \cite{Ru:Vojta:2016}.  

As some evidence for the utility of this reciprocity, between techniques in $\mathrm{K}$-stability and Diophantine arithmetic geometry, the work \cite{Grieve:Divisorial:Instab:Vojta} demonstrates how the concept K-instability, for $\QQ$-Fano varieties with canonical singularieties, implies instances of Vojta's Main Conjecture.  Additionally, \cite[Section 6]{Grieve:toric:gcd:2019} shows that, for the case of $\PP^1 \times \PP^1$ blown-up along a torus invariant point, calculations with Newton-Okounkov bodies allow for an improvement of \cite[Corollary 1.12]{Ru:Vojta:2016} and \cite[Proposition 8]{Guo:Wang:2019}.

As it turns out, a key point that underlies all of this,
is the fact that the expectation of Duistermaat-Heckman measure for filtered linear series may be understood, via the theory of concave transforms, in terms of Newton-Okounkov bodies and certain limit expectations, which are closely related to the theory of Chow weights for polarized projective varieties.  This is discussed in \cite{Grieve:2018:autissier}, \cite{Grieve:toric:gcd:2019}, \cite{Grieve:Divisorial:Instab:Vojta} and \cite{Grieve:chow:approx}.  These works provide additional context and motivation for exposing the main results from \cite{Boucksom:Chen:2011} and \cite{Boucksom:Kuronya:Maclean:Szemberg:2015}.  This is our main purpose here.  (See Theorems \ref{DH:Expect} and \ref{DH:restricted:volume:thm}.)

\section{Filtered linear series}\label{Preliminiaries:Filtered:Algebras}

In this section, we fix our conventions about filtrations of the section ring
$$
R = R(X,L) := \bigoplus_{m\geq 0} \H^0(X,mL) = \bigoplus_{m \geq 0} R_m \text{,}
$$
for $L$ a big line bundle on $X$, an irreducible projective variety over $\kk$.
Our conventions and notations closely follow those of \cite[Definition 1.3]{Boucksom:Chen:2011} and \cite{Boucksom:Kuronya:Maclean:Szemberg:2015}.  

Specifically, the $\RR$-filtrations 
$$\mathcal{F}^\bullet = \mathcal{F}^\bullet R = \left( \mathcal{F}^t R_m \right)_{t \in \RR, m \in \ZZ_{\geq 0}}$$ 
that we consider have the properties that
\begin{enumerate}
\item[(i)]{
if $m \in \ZZ_{\geq 0}$, then
\begin{itemize}
\item{$\mathcal{F}^t R_m \subseteq \mathcal{F}^{t'} R_m$, for all $t,t'\in \RR$ with $t \geq t'$;}
\item{$\mathcal{F}^t R_m = \mathcal{F}^{t - \epsilon} R_m$, for all $t \in \RR$ and all $0 < \epsilon \ll 1$; and}
\end{itemize}
}
\item[(ii)]{
if $m,m' \in \ZZ_{\geq 0}$ and $t,t'\in\RR$, then
$$ \mathcal{F}^t R_m \cdot \mathcal{F}^{t'}R_m \subseteq \mathcal{F}^{t + t'} R_{m+m'}\text{.}$$ 
}
\end{enumerate}

Following the terminology of \cite{Boucksom:Chen:2011}, such filtrations are \emph{decreasing}, \emph{left-continuous}
and \emph{multiplicative}.

Put
$$
a_j(mL) := \inf \{ t \in \RR : \operatorname{codim}_{\kk} \mathcal{F}^t R_m \geq j + 1 \}\text{.}
$$
These are the \emph{vanishing numbers} of $R$, with respect to $\mathcal{F}^\bullet$.  Evidently
$$
 a_0(mL) \leq \dots \leq a_{n_m}(mL) 
$$
for 
$
n_m := \dim R_m - 1.
$

Furthermore, the non-increasing left-continuous step functions
$$
t \mapsto \dim_\kk \mathcal{F}^t R_m\text{, }
$$
for $m \in \ZZ_{\geq 0}$, satisfy the condition that
$$
\dim_\kk \mathcal{F}^t R_m = \dim R_m -  j
$$
if and only if 
$$
t \in \ ]a_{j-1}(mL), a_j(mL)].
$$
Here, we have also set
$$
a_{-1}(mL) = - \infty
$$
and
$$
a_{n_m+1}(mL) = + \infty \text{.}
$$
It follows, in the sense of distributions, that
$$
 \frac{d}{dt} \dim \mathcal{F}^t R_m = - \sum_{j=0}^{n_m} \delta_{a_j(mL)}  \text{.}
$$

For later use, we set
$$
a_{\min}(mL) :=  a_0(mL)\text{, }
$$
$$
a_{\max}(mL) := a_{n_m}(mL) 
$$
and
$$
s(mL,\mathcal{F}^\bullet) := \sum_{a_j(mL) > 0} a_j(mL) \text{.}
$$
Here, by convention, we set 
$$
s(mL, \mathcal{F}^\bullet) = 0
$$
in case that 
$$
a_{\max}(mL) \leq 0 \text{.}
$$

Our main interest here is to study those filtrations which are \emph{pointwise bounded below} and \emph{linearly bounded above}, in the senes of Definition \ref{linearly:bounded:above}.  

\begin{defn}[See {\cite[Definition 1.3]{Boucksom:Chen:2011}}]\label{linearly:bounded:above}  An $\RR$-filtration $\mathcal{F}^\bullet$ of $R$ is \emph{pointwise bounded below} if 
$$\mathcal{F}^t R_m = R_m\text{,}$$ for all $t \ll - 1$ and all $m \in \ZZ_{\geq 0}$.  It is \emph{linearly bounded above} if there exists a constant $C > 0$ such that
$$
a_{\max}(mL) \leq C m \text{,}
$$
for all $m \in \ZZ_{\geq 0}$.
\end{defn}

A natural class of linearly bounded above filtrations arise via orders of vanishing along divisors.  

\begin{example}\label{orders:vanishing}
Let $L$ be a big line bundle on an irreducible normal projective variety $X$.  Let $E$ be a nonzero effective Cartier divisor on some normal proper model of $X$.  
Each such Cartier divisor $E$ over $X$ determines a linearly bounded above $\RR$-filtration 
$$\mathcal{F}^\bullet = \mathcal{F}^\bullet R$$
of the section ring
$$
R := \bigoplus_{m\geq 0} \H^0(X,mL) = \bigoplus_{m \geq 0} R_m \text{.} 
$$
Indeed, such filtrations have the form
$$
\mathcal{F}^t R_m := \H^0(Y, m\pi^* L -  \lceil t E \rceil  )\text{, }
$$
for 
$$\pi \colon Y \rightarrow X$$
some normal proper model with 
$$E \subseteq Y\text{.}$$  
Evidently, such filtrations are pointwise bounded below.
That such filtrations are linearly bounded from above follows from the fact that
$$
\mathcal{F}^t R_m = 0
$$
for all $t > \gamma_{\mathrm{eff}} m$, where
$$\gamma_{\mathrm{eff}} := \inf \{ t : \pi^*L - t E  \text{ is pseudoeffective} \} < \infty \text{.}$$
\end{example}

Returning to more general considerations, given a pointwise bounded below and linearly bounded above filtration 
$\mathcal{F}^\bullet$,
of the section ring
$$
R := \bigoplus_{m\geq 0} \H^0(X,mL) = \bigoplus_{m \geq 0} R_m \text{,} 
$$
of a big line bundle $L$ on an irreducible projective variety $X$,
we set
$$
R^t_m := \mathcal{F}^{tm} R_m.
$$
In this way, we obtain graded subalgebras
$$
R^t_\bullet := \bigoplus_{m \geq 0} R^t_m \subseteq R
$$
for each $t \in \RR$.

We also make explicit note of the following special case of \cite[Lemma 1.4]{Boucksom:Chen:2011}.  We include a proof for the sake of completeness.

\begin{lemma}[{\cite[Lemma 1.4]{Boucksom:Chen:2011}}]\label{a:max:lemma}
Let $\mathcal{F}^\bullet$ be a linearly bounded above filtration of the section ring 
$$
R := \bigoplus_{m\geq 0} \H^0(X,mL) = \bigoplus_{m \geq 0} R_m \text{,} 
$$
for $L$ a big line bundle on an irreducible projective variety $X$.
Then
$$
a_{\max}(R,\mathcal{F}^\bullet) = \lim_{m \to \infty} \frac{ a_{\max}(mL) }{ m } = \sup_{m \geq 1} \frac{a_{\max}(mL)}{m}\text{.}
$$
\end{lemma}

\begin{proof}
Observe that the quantities $a_{\max}(mL)$ are super-additive in $m$, for $m \gg 1$.  In particular, they have the property that
$$
a_{\max}(mL) + a_{\max}(nL) \leq a_{\max}((m+n)L) \text{,}
$$
for all $m, n \in \ZZ_{\geq 1}$.  So, the desired results follow because of the fact that, for super-additive sequences, the equality
$$
a_{\max} = \lim_{m \to \infty} \frac{a_m}{m} = \sup_{m \geq 1} \frac{a_m}{m} 
$$
holds true.
\end{proof}

\section{Newton-Okounkov bodies and concave transforms}\label{OK:bodies:concave:transform}

Let 
$\mathcal{F}^\bullet$
be a pointwise bounded below and linearly bounded above filtration of the section ring
$$
R := \bigoplus_{m\geq 0} \H^0(X,mL) = \bigoplus_{m \geq 0} R_m \text{,} 
$$
of $L$ a big line bundle on a $d$-dimensional irreducible projective variety $X$.   Our goal here, is to recall the construction of the \emph{concave transform} of $\mathcal{F}^\bullet$ with respect to a given admissible flag of subvarieties.  

In order to do so, we recall briefly, the basic theory of Newton-Okounkov bodies, from \cite{Lazarsfeld:Mustata:2009}, and adopt similar notation.

Let 
$$V_\bullet = \bigoplus_{m \geq 0} V_m$$ 
be a graded linear series of $L$ and
assume that $V_\bullet$ contains an \emph{ample series}.  Then $V_\bullet$ is a graded $\kk$-subalgebra of $R(X,L)$ and  \cite[Definition 2.9 Condition (C) point (i)]{Lazarsfeld:Mustata:2009} holds true.  

Now, upon fixing an \emph{admissible flag} of irreducible subvarieties
\begin{equation}\label{eqn:6}
X_\bullet \colon X = X_0 \supsetneq X_1 \supsetneq \hdots \supsetneq X_d = \{\mathrm{pt} \}
\end{equation}
we may associate to  $V_\bullet$, via the valuation like function
\begin{equation}\label{eqn:7}
\nu_{X_\bullet} \colon V_\bullet \setminus \{0\} \rightarrow \ZZ^d
\end{equation}
 that is determined by $X_\bullet$, the graded semigroup
$$
\Gamma(V_\bullet) := \{(\nu_{X_\bullet}(\sigma),m) : 0 \not = \sigma \in V_m \text{ and } m \geq 0 \}  \subseteq \ZZ_{\geq 0}^{d+1} \text{.}
$$

To construct the valuation like functions \eqref{eqn:7}, there is no loss in generality by assuming that each subvariety $X_{i+1}$,  for $i = 0,\dots,d-1$, is a Cartier divisor in $X_i$.  Next, given a nonnegative integer $m$, with 
$$\H^0(X,mL) \not = 0\text{,}$$ 
let
$$
0 \not = \sigma \in \H^0(X,\Osh_X(mL))
$$
be a nonzero section.  Put
$$
v_1 = v_1(\sigma) := \operatorname{ord}_{X_1}(\sigma) \text{.}
$$

Inductively, and by choosing a local equation for $X_{j+1}$ in $X_j$, we may construct a nonvanishing section
$$
0 \not = \sigma_j \in \H^0\left(X_j,mL |_{X_j} \otimes \bigotimes_{i=1}^j \Osh_{X_{i-1}}(-v_i X_i)|_{X_{j}}\right) \text{;}
$$
we then set
$$
v_{j+1} = v_{j+1}(\sigma_j) := \operatorname{ord}_{X_{j+1}}(\sigma_j) \text{.}
$$
In this way, each nonzero section
$$
0 \not = \sigma \in \H^0(X,mL) \text{,}
$$
determines a well-defined $d$-tuple of integers
\begin{equation}\label{eqn:8}
v(\sigma) = (v_1,\dots,v_d) \in \ZZ^d
\end{equation}
and depending on the fixed choice of admissible flag \eqref{eqn:6}.  These integer vectors \eqref{eqn:8} allow for the definition of the valuation like functions \eqref{eqn:7}.

In this context, the closed convex cone
$$
\Sigma(V_\bullet) \subseteq \RR^{d+1}\text{, }
$$
which is generated by $\Gamma(V_\bullet)$, has a compact convex basis
$$
\Delta(V_\bullet) := \Sigma(V_\bullet) \bigcap \left( \{1 \} \times \RR^d \right)\text{.}
$$
This is the \emph{Newton-Okounkov body} of $V_\bullet$.  It depends on the fixed choice of admissible flag $X_\bullet$.

On the other hand, following \cite[Definition 1.6]{Boucksom:Chen:2011}, the \emph{concave transform} of $\mathcal{F}^\bullet$ is the concave function
$$
G_{\mathcal{F}^\bullet} \colon \Delta(L) \rightarrow [-\infty, \infty[ \text{,}
$$
which is defined by
$$
x \mapsto \sup \left\{t \in \RR : x \in \Delta \left(R^t_\bullet \right)\right\}\text{.}
$$
Since $\mathcal{F}^\bullet$ is assumed to be pointwise bounded below, it follows that $G_{\mathcal{F}^\bullet}$ takes finite values on $\Delta(L)^\circ$, see \cite[Lemma 1.7]{Boucksom:Chen:2011}, and is continuous there too.
 Finally, via the concave transform, as in \cite[Definition 1.9]{Boucksom:Chen:2011}, the \emph{filtered Newton-Okounkov body} is the compact convex subset
$$
\Delta(\mathcal{F}^\bullet ) := \left\{ (x,t) \in \Delta(L) \times \RR \text{, } 0 \leq t \leq G_{\mathcal{F}^\bullet}(x) \right\} \subseteq \RR^{d+1} \text{.}
$$ 

\section{Proof of Theorem \ref{DH:Expect}}\label{DH:Proof}

In this section, we establish Theorem \ref{DH:Expect}.  Our  arguments follow closely those of \cite[Proof of Theorem 1.11]{Boucksom:Chen:2011} and \cite[Proof of Theorem 2.20]{Boucksom:Kuronya:Maclean:Szemberg:2015}.

\begin{proof}[Proof of Theorem \ref{DH:Expect}]
Let 
$$d := \dim X$$ 
and put 
$$
\mu_m := \frac{h^0(X,mL)}{m^d} \nu_m \text{.}
$$

Then $- \mu_m$ is the distributional derivative of the non-increasing left continuous step function
$$
g_m(t) := m^{-d} \dim \mathcal{F}^{mt} R_m = m^{-d} \dim R^t_m.
$$
Set
$$
g(t) := \Vol \left( \Delta(R_\bullet^t) \right)
$$
and recall that the graded subalgebras 
$$
R^t_\bullet := \bigoplus_{m \geq 0} R^t_m = \bigoplus_{m \geq 0} \mathcal{F}^{tm} R_m\text{,}
$$
for 
$$t < a_{\max}(R,\mathcal{F}^\bullet)\text{,}$$ 
contain an ample series \cite[Lemma 1.6]{Boucksom:Chen:2011}.

Thus, by the theory of Newton-Okounkov bodies, \cite[Lemma 2.12 and Theorem 2.13]{Lazarsfeld:Mustata:2009}, we obtain the convergence 
$$ 
\lim_{m\to \infty} g_m(t) = g(t).
$$
Furthermore, observe that the inequality 
$$
0 \leq g_m(t) \leq m^{-d} h^0(X,mL)
$$
is uniformly bounded.

Thus, by dominated convergence, we obtain the convergence
$$g_m \to g$$
in $L^1_{\mathrm{loc}}(\RR)$.  In particular
$$
- \mu_m = \frac{d}{dt}(g_m)  \to \frac{dg}{dt} 
$$
in the sense of distributions.

We now interpret these observations in terms of the concave transform $G_{\mathcal{F}^\bullet}$.  Let $\lambda$ denote the restriction of the Lebesgue measure to the interior of the Newton-Okounkov body $\Delta(L)$ and consider its pushforward with respect to the concave transform
$$
\mu := \left( G_{\mathcal{F}^\bullet} \right)_* \lambda \text{.}
$$
To complete the proof of Theorem \ref{DH:Expect}, it remains to show that
$$
\frac{dg}{dt}  = - \mu.
$$

To this end, first observe that
$$
\lim_{s \to t^-} g(s) = \lambda \left( \{G_{\mathcal{F}^\bullet} \geq t \} \right) \text{.}
$$
On the other hand 
$$
h(t) := \lambda \left( \{ G_{\mathcal{F}_\bullet} \geq t \} \right)
$$
and the discontinuity locus of $g(t)$ is at most countable.  Thus, we obtain the equality of distributions
$$
g = h \text{.}
$$

Then note that,
on $\RR$, we have 
$$
h(t) = \mu \left( \{x \in \RR : x \geq t \} \right)\text{;}
$$
it then follows, from integration theory, that
$$
\frac{dh}{dt} = - \mu\text{.}
$$
Moreover, since the measures $\mu_m$ and $\nu_m$ are related as
$$
\mu_m = \frac{h^0(X,mL)}{m^d} \nu_m\text{,}
$$
it follows that
$$
\mu = \lim_{m \to \infty} \mu_m  = \frac{\Vol_X(L)}{g!} \lim_{m \to \infty} \nu_m = \frac{\Vol_X(L)}{g!} \nu \text{.}
$$
This establishes weak convergence of measures 
$$\nu_m \to \nu\text{.}$$

Finally, 
similar to \cite[Corollary 1.13]{Boucksom:Chen:2011}, we observe how the Euclidean volume of the filtered Newton-Okounkov bodies relate to the Duistermaat-Heckman measures.

First of all, observe that
the expectation $\mathbb{E}(\nu)$, of the Duistermaat-Heckman measure $\nu$, determined by $\mathcal{F}^\bullet$, may be described as
$$
\mathbb{E}(\nu) = \frac{d!}{\Vol_X(L)} \int_0^\infty t \cdot \mathrm{d} \mu(t) \text{.}
$$
Indeed, in our present notation
$$
\frac{ s(mL,\mathcal{F}^\bullet)}{m^{d+1}} = \int_0^\infty t \cdot \mathrm{d}(\mu_m) \text{.}
$$
Furthermore, 
$$
\Vol(\Delta(\mathcal{F}^\bullet R)) = \int_0^\infty t \cdot \left( G_{\mathcal{F}^\bullet} \right)_* \lambda
$$
and, by Lemma \ref{a:max:lemma}, we have that
$$
a_{\max} := a_{\max}(R,\mathcal{F}^\bullet)  = \lim_{m \to \infty} \frac{a_{\max}(mL)}{m} \text{.}
$$
Finally, note that since 
$$G_{\mathcal{F}^\bullet} \leq a_{\max} \text{,}$$
the measures $\mu_m$ and $(G_{\mathcal{F}^\bullet})_* \lambda$ are both supported on the half open interval 
$$]-\infty, a_{\max}]\text{.}$$

Thus, since, the measures $\nu_m$ and $\mu_m$ both have uniform compact supports, the respective weak convergences 
$$\nu_m \to \nu$$ 
and 
$$\mu_m \to \mu$$ 
imply the desired conclusion
$$
\mathbb{E}(\nu) = \lim_{m \to \infty} \mathbb{E}(\nu_m) \text{.}
$$
\end{proof}

\section{Interpretation in terms of restricted volume functions}\label{restricted:volume:functions}
We now establish, closely following \cite[Theorem 2.24]{Boucksom:Kuronya:Maclean:Szemberg:2015}, Theorem \ref{DH:restricted:volume:thm}.  

\begin{proof}[Proof of Theorem \ref{DH:restricted:volume:thm}]
 Recall, that we are given a normal proper model 
$$
\pi \colon Y \rightarrow X 
$$
with the property that 
$$E \subseteq Y\text{.}$$  
Denote by 
$$h^0(X|E, L) = h^0(Y|E, \pi^* L) \text{,}$$ 
the rank of the restriction map
$$
\H^0(Y,\pi^* L) \rightarrow \H^0(E, \pi^*L|E).
$$
Let $d$ denote the dimension of $X$.
The restricted volume of $L$ along $E$, then takes the form
\begin{align*}
\Vol_{X|E} (L) & := \Vol_{Y|E}(\pi^* L) \\
 & = \limsup_{m \to \infty} \frac{(d-1)! }{m^{d-1 }} h^0(X|E,L) 
\text{.}
\end{align*}

Next, similar to \cite[Section 4.2]{Boucksom:Favre:Jonsson:2009}, in order to determine the nature of the restricted volume function $\Vol_{X|E} (L)$, we may replace $Y$ with some non-singular model 
$$
\pi' \colon Y' \rightarrow Y
$$
and $E$ with its strict transform $E'$ with respect to $\pi'$.  

Indeed, this follows in light of the relation
$$
\Vol_{Y|E}(\pi^*L) = \Vol_{Y'|E'}(L') \text{, }
$$
for $L'$ the pullback of $L$ to $Y'$ under the composition
$$
Y' \xrightarrow{\pi'} Y \xrightarrow{\pi} X \text{.}
$$
Note that, by normality of $X$, the fibres of such models over $X$ are connected.

Now, recall that the filtration
$
\mathcal{F}^\bullet 
$
is described as
$$
\mathcal{F}^t R_m = \H^0(Y, m \pi^* L - \lceil t E \rceil) \text{.}
$$
It is thus evident that 
$$
a_{\max}(R,\mathcal{F}^\bullet) = \sup \{t > 0 : \pi^* L - t E \text{ is big} \} \text{.}
$$ 

Next, set 
$$v = E$$ 
and
$$
\Vol(L, v \geq t) := \lim_{m \to \infty} \frac{d!}{m^d} h^0(Y, m \pi^* L - m t E) \text{.}
$$
In this notation
$$
\nu = \frac{\Vol(L , x \geq t) }{ \Vol(L) } \mathrm{d}t\text{.}
$$

Finally, by the continuity of the volume function, the limit measure $\nu$ is absolutely continuous with respect to Lebesgue measure on the interval 
$$]-\infty, a_{\max}(R, \mathcal{F}^\bullet)[ \text{.}$$  
Furthermore, $\nu$ is the weak derivative of
$$
- \frac{\Vol(L, x \geq t)}{\Vol(L) } = - \frac{\Vol(\pi^*L - t E) }{ \Vol(L) } \text{.}
$$

The result then follows by the differentiability property of the volume function \cite[Theorem A]{Boucksom:Favre:Jonsson:2009}.  Indeed, recall that the restricted volume functions are related as
$$
\Vol_{X|E}(L) = \Vol_{Y|E}(\pi^*L) = \Vol_{Y'|E'}(L') \text{.}
$$
\end{proof}

\subsection*{Acknowledgements}   
It is my pleasure to thank colleagues for their interest, comments and discussions on related topics.   Further, I thank the Natural Sciences and Engineering Research Council of Canada for their support through my grants DGECR-2021-00218 and RGPIN-2021-03821.   Finally, I thank anonymous referees for their comments, careful reading and suggestions.  They helped to improve upon an earlier version of this article.

\providecommand{\bysame}{\leavevmode\hbox to3em{\hrulefill}\thinspace}
\providecommand{\MR}{\relax\ifhmode\unskip\space\fi MR }
\providecommand{\MRhref}[2]{%
  \href{http://www.ams.org/mathscinet-getitem?mr=#1}{#2}
}
\providecommand{\href}[2]{#2}

\end{document}